\documentclass[a4paper,12pt]{article}
\usepackage{amsmath}
\usepackage{amssymb}
\usepackage{latexsym}
\usepackage{amsthm}
\newtheorem{Thm}{Theorem}[section]
\newtheorem{Cor}[Thm]{Corollary}

\newtheorem{Lem}[Thm]{Lemma}
\newtheorem{Prop}[Thm]{Proposition}

\theoremstyle{definition}
\newtheorem{Def}[Thm]{Definition}

\newtheorem{Rem}[Thm]{Remark}
\newtheorem{Exa}[Thm]{Example}

\begin{document}

\title{On the range of random walk on graphs satisfying a uniform condition}
\author{Kazuki Okamura\footnote{Graduate School of Mathematical Sciences, The University of Tokyo; e-mail : \texttt{kazukio@ms.u-tokyo.ac.jp}  \, \, This work was supported by Grant-in-Aid for JSPS Fellows.}}
\date{}
\maketitle

\begin{abstract}
We consider the range of random walks up to time $n$, $R_{n}$, on graphs satisfying a uniform condition.
This condition is characterized by potential theory. 
Not only all vertex transitive graphs
but also many non-regular graphs satisfy the condition.
We show certain weak laws of $R_{n}$ from above and  below.
We also show that 
there is a graph such that it satisfies the condition and a sequence of the mean of $R_{n}/n$ fluctuates. 
By noting the construction of the graph,  
we see that under the condition, the weak laws are best in a sense.
\end{abstract}

\section{Introduction}
The range of random walk $R_{n}$ is simply the number of sites which the random walk visits up to time $n$.
One of the most fundamental problems is whether the process $\{R_{n}\}_{n}$ satisfies law of large numbers.
Dvoretzky and Erd\"os \cite{DE}, Spitzer \cite{Sp} considered the ranges of random walks on $\mathbb{Z}^{d}$
and derived strong law of large numbers.
They used the spacial homogeneity of $\mathbb{Z}^{d}$ heavily. 
We may need to take alternative techniques to consider the range of walks on  graphs which do \textit{not} have such spacial homogeneity.

In this paper we consider the range of random walk on graphs satisfying a uniform condition $(U)$. 
See Definition 1.1 for the definition of the uniform condition. 
This condition is characterized by potential theory, specifically,
effective resistances.
Not only all vertex transitive graphs 
but also some \textit{non-regular} graphs
satisfy $(U)$.
See Section 4 for detail.
We state certain weak laws of $R_{n}$ from above and  below in Theorem 1.2.
Under a stronger assumption, certain strong laws holds for $R_{n}$. 
In Theorem 1.3, 
we state the existence of a graph such that it satisfies $(U)$
and a sequence of the mean of $R_{n}/n$ fluctuates.
This construction shows that under $(U)$, the two convergences are best in a sense.

Now we describe the settings.
Let $(X, \mu)$ be an weighted graph.
That is, $X$ is an infinite weighted graph and $X$ is endowed with a weight $\mu_{xy}$,
which is a symmetric nonnegative function on $X \times X$
such that $\mu_{xy} > 0$ if and only if $x$ and $y$ are connected.
We write $x \sim y$ if $x$ and $y$ are connected by an edge.
Let $\mu_{x} = \sum_{y \in X} \mu_{xy}$, $x \in X$.
Let $\mu(A) = \sum_{x \in A} \mu_{x}$ for $A \subset X$.

In this paper we assume that $\sup_{x \in X}\deg(x) < +\infty$ and $0 < \inf_{x, y \in X, x \sim y} \mu_{xy} \le \sup_{x, y \in X, x \sim y} \mu_{xy} < +\infty$.
\textit{Whenever we do not refer to weights,
we assume that $\mu_{xy} = 1$ for any $x \sim y$.}

Let $\{S_{n}\}_{n \geq 0}$ be a Markov chain  on $X$
whose transition probabilities are given by $P(S_{n+1} = y | S_{n} = x) = \mu_{xy}/\mu_{x}$, $n \geq 0$, $x, y \in X$.
We write $P = P_{x}$ if $P(S_{0} = x) = 1$. 
We say that $(X, \mu)$ is recurrent (resp. transient) 
if $(\{S_{n}\}_{n \geq 0}, \{P_{x}\}_{x \in X})$ is recurrent (transient).
Let the random walk range $R_{n} = |\{S_{0}, \dots, S_{n-1}\}|$.

Let $T_{A} = \inf \{n \geq 0 : S_{n} \in A\}$ and 
$T_{A}^{+} = \inf \{n \geq 1 : S_{n} \in A\}$ for $A \subset X$.
For $x,y \in X$, $n \geq 0$ and  $B \subset X$, 
let $p_{n}^{B}(x, y) = P_{x}(S_{n} = y,  T_{B^{c}} > n)/\mu_{y}$ and $g^{B}(x, y) =  \sum_{n \ge 0} p_{n}^{B}(x, y)$.
Let $p_{n}(x, y) = p^{X}_{n}(x,y)$
and $g(x, y) = g^{X}(x,y)$.

Let $F_{1} = \inf_{x \in X} P_{x}(T_{x}^{+} < +\infty)$ and
$F_{2} = \sup_{x \in X} P_{x}(T_{x}^{+} < +\infty)$.

Let $d$ be the graph metric on $X$.
Let $B(x, n) = \{y \in X : d(x,y) < n\}$, $x \in X$, $n \in \mathbb{N}_{\ge 1}$.
Let $V(x, n) = \mu(B(x, n))$.
Let $\mathcal{E}(f, f) = \frac{1}{2} \sum_{x, y \in X, x\sim y} (f(x)-f(y))^{2}\mu_{xy}$ for $f : X \to \mathbb{R}$.
Let us define the effective resistance  by 
$R_{\text{eff}}(A, B)^{-1} = \inf \{\mathcal{E}(f,f) : f|_{A} = 1,  f|_{B} = 0\}$ for $A, B \subset X$ with $A \cap B = \emptyset$.

Let $\rho(x,n) = R_{\textrm{eff}}(\{x\}, B(x,n)^{c})$, $x \in X, n \in \mathbb{N}$.
Let $\rho(x)  = \lim_{n \to \infty} \rho(x, n)$. 
If $(X, \mu)$ is recurrent (resp. transient), then, $\rho(x) = +\infty$ (resp. $\rho(x) < +\infty$) for any $x \in X$.

Now we define a uniform condition for weighted graphs.
\begin{Def}[uniform condition]
We say that an weighted graph $(X, \mu)$ satisfies $(U)$
if $\rho(x, n)$ converges \textit{uniformly} to $\rho(x)$, $n \to \infty$.
\end{Def}

Not only vertex transitive graphs 
(e.g. $\mathbb{Z}^{d}$, the $M$-regular tree $T_{M}$, Cayley graphs of groups) 
but also some non-regular  graphs
(e.g. graphs which are roughly isometric with $\mathbb{Z}^{d}$, Sierpi\'nski gasket or carpet)
satisfy $(U)$ if all weights are equal to $1$.
See Section 4 for detail. 

Now we describe the main results.

\begin{Thm}
Let $(X, \mu)$ be an weighted graph satisfying $(U)$.
Then, for any $x \in X$ and any $\epsilon > 0$,
we have that
\begin{equation} 
\lim_{n \to \infty} P_{x}(R_{n} \ge n(1-F_{1} + \epsilon)) = 0,  \tag{1.1}
\end{equation}
and, 
\begin{equation} 
\lim_{n \to \infty} P_{x}(R_{n} \le n(1-F_{2} - \epsilon)) = 0. \tag{1.2}
\end{equation}
These convergences are uniform with respect to $x$. 
The convergence in $(1.1)$ is exponentially fast.
\end{Thm}

If $(X, \mu)$ satisfies an assumption which is stronger than $(U)$,  
then, certain strong laws hold for $R_{n}$, that is,
\[ 1-F_{2} \le \liminf_{n \to \infty} \frac{R_{n}}{n} \le \limsup_{n \to \infty} \frac{R_{n}}{n} \le 1-F_{1}, \, P_{x}\text{-a.s}. \] 
See Corollary 2.3 for detail.  

\begin{Thm}
There exists an infinite weighted graph $(X, \mu)$ with a reference point $o$ which satisfies $F_{1} < F_{2}$, $(U)$, 
\begin{equation} 
\liminf_{n \to \infty}\frac{E_{o}[R_{n}]}{n} = 1-F_{2}, \text{ and, } \limsup_{n \to \infty}\frac{E_{o}[R_{n}]}{n} = 1-F_{1}. \tag{1.3}
\end{equation}
\end{Thm}

\begin{Rem}
(i) If $X$ is vertex transitive,
then,
$F_{1} = F_{2}$ and hence $R_{n}/n \to 1-F_{1} \in [0,1]$ in probability.
On the other hand, 
by noting Theorem 1.3, 
there exists an infinite weighted graph $(X, \mu)$ with a reference point $o$ which satisfies $(U)$ and 
$R_{n}/n$ does \textit{not} converge to any $a \in [0,1]$ in probability under $P_{o}$.\\
(ii) If we replace $F_{1}$ (resp. $F_{2}$) with a real number larger than  $F_{1}$ (resp. smaller than $F_{2}$), 
(1.1) (resp. (1.2)) fails for an weighted graph in Theorem 1.3.
In this sense, the convergences (1.1) and (1.2) are best.
\end{Rem}

The main difficulty of the proof of Theorem 1.2  is that $P_{x} \ne P_{y}$ can happen for $x \ne y$.
On the other hand, we use the fact in order to show Theorem 1.3.


\section{Proof of Theorem 1.2}

First, we show the following lemma.
\begin{Lem}
Let $(X, \mu)$ be an weighted graph satisfying $(U)$.
Then, \[ \lim_{n \to \infty} \sup_{x \in X} P_{x}(n < T_{x}^{+} < +\infty) = 0. \]
\end{Lem}

\begin{proof}
By Kumagai \cite{Kum} Theorem 1.14, 
$\rho(x, n)^{-1} = \mu_{x} P_{x}(T_{x}^{+} > T_{B(x,n)^{c}})$, $x \in X$, $n \ge 1$.
Letting $n \to \infty$,
we have $\rho(x)^{-1} = \mu_{x} P_{x}(T_{x}^{+} = +\infty)$.

Since $\rho(x, 1)^{-1} = \mu_{x}$,
\begin{align*} 
P_{x}(T_{B(x,n)^{c}} < T_{x}^{+} < +\infty) 
&= \mu_{x}^{-1}(\rho(x, n)^{-1} - \rho(x)^{-1})\\
&\le \mu_{x}(\rho(x)- \rho(x, n)).
\end{align*}

Since $\mu_{x} \le \sup_{y \in X}\deg(y) \sup_{y, z \in X, y \sim z} \mu_{yz} < +\infty$ and $(X, \mu)$ satisfies (U),
we see that
\begin{equation}
\lim_{n \to \infty} \sup_{x \in X} P_{x}(T_{B(x,n)^{c}} < T_{x}^{+} < +\infty) = 0. \tag{2.1}
\end{equation}

Since $\sup_{x}\deg(x) < +\infty$ and $\sup_{y, z \in X, y \sim z}\mu_{yz} < +\infty$,
we have that 
$\sup_{x \in X} V(x, n) < +\infty$, $n \ge 1$.
Since $\rho(x, n)^{-1} \ge \inf_{y,z \in X, y \sim z} \mu_{yz}/n > 0$,
we have that 
$\sup_{x \in X} \rho(x, n) < +\infty$, $n \ge 1$.

Thus we can let $f(n) = \sup_{x \in X} \rho(x, n) \sup_{x \in X} V(x, n)$, $n \ge 1$.

By \cite{Kum} Lemma 3.3(v), 
\[ P_{x}(T_{B(x,n)^{c}} \ge nf(n)) \le  \frac{E_{x}[T_{B(x,n)^{c}}]}{nf(n)} \le \frac{\rho(x, n)V(x, n)}{nf(n)} \le \frac{1}{n}.\]
Hence, 
\begin{equation}
\lim_{n \to \infty} \sup_{x \in X} P_{x}(T_{B(x,n)^{c}} \ge nf(n)) = 0. \tag{2.2}
\end{equation}

We have that
\[ P_{x}(nf(n) < T_{x}^{+} < +\infty) 
\le P_{x}(T_{B(x,n)^{c}} < T_{x}^{+} < +\infty) + P_{x}(T_{B(x,n)^{c}} \ge nf(n)). \]
By noting (2.1) and (2.2),
we have that 
$$\lim_{n \to \infty} \sup_{x \in X} P_{x}(nf(n) < T_{x}^{+} < +\infty) = 0.$$
This completes the proof of Lemma 2.1.
\end{proof}

Let $Y_{i, j}$ be the indicator function of 
$\{S_{i} \ne S_{i+k} \text{ for any } 1 \le k \le j\}$.
Let $Y_{i, \infty}$ be the indicator function of $\{S_{i} \ne S_{i+k} \text{ for any } k \ge 1\}$.

\begin{proof}[Proof of Theorem 1.2]
We show this assertion in a manner which is partially similar to the proof of Theorem 1 in Benjamini, Izkovsky and Kesten \cite{BIK}.
However $P_{x} \ne P_{y}$ can happen for $x \ne y$ and hence the random variables $\{Y_{k+aM, M}\}_{a \in \mathbb{N}}$  are \textit{not} necessarily independent.
The details are different from  the proof of Theorem 1 in \cite{BIK}.

First, we will show (1.1).
Let $\epsilon > 0$.
Let $M$ be a positive integer such that $\sup_{x \in X}P_{x}(M < T_{x}^{+} < +\infty) < \epsilon/4$.
We can take such $M$ by Lemma 2.1.

By considering a last exit decomposition (as in \cite{BIK}), 
\[ R_{n} = 1 + \sum_{i=0}^{n-2} Y_{i, n-1-i}
\le M + \sum_{i=0}^{n-1-M} Y_{i, n-1-i}
\le  M + \sum_{i=0}^{n-1-M} Y_{i, M}. \]

Hence for  $n > 2M/\epsilon$,
\begin{align*} 
P_{x}(R_{n} \ge n(1-F_{1} + \epsilon)) 
&\le P_{x}\left(\sum_{i=0}^{n-1-M} Y_{i, M} > n\left(1-F_{1} + \frac{\epsilon}{2}\right)\right) \\
&= P_{x}\left(\sum_{a=0}^{M} \sum_{i \equiv a \!\!\!\!\mod(M+1)}Y_{i, M} > n\left(1-F_{1} + \frac{\epsilon}{2}\right)\right)\\
&\le \sum_{a=0}^{M}P_{x}\left(\sum_{i \equiv a \!\!\!\!\mod(M+1)}Y_{i, M} > \frac{n}{M+1}\left(1-F_{1} + \frac{\epsilon}{2}\right)\right).
\end{align*}

Therefore it is sufficient to show that for each $a\in \{0, 1, \dots, M\}$, 
\[ P_{x}\left(\sum_{i \equiv a \!\!\!\!\mod(M+1)}Y_{i, M} > \frac{n}{M+1}\left(1-F_{1} + \frac{\epsilon}{2}\right)\right) \to 0, \, n \to \infty, \, \text{exponentially fast}. \tag{2.3} \]

For any $t > 0$, we have that
\[ P_{x}\left(\sum_{i \equiv a \!\!\!\!\mod(M+1)}Y_{i, M} > \frac{n}{M+1}\left(1-F_{1} + \frac{\epsilon}{2}\right)\right) \]
\begin{equation} \le \exp\left(-t\frac{n}{M+1}\left(1-F_{1} + \frac{\epsilon}{2}\right)\right)
E_{x}\left[\exp\left(t\sum_{i \equiv a \!\!\!\!\mod(M+1)}Y_{i, M} \right)\right]. \tag{2.4}
\end{equation}

By using the Markov property of $\{S_{n}\}_{n}$,
\begin{align*}
E_{x}\left[\exp\left(t\sum_{i \equiv a \!\!\!\!\mod(M+1)}Y_{i, M} \right)\right]
&= E_{x}\left[ \prod_{i \equiv a \!\!\!\!\mod(M+1)} \exp(tY_{i, M})\right]\\
&\le \left(\sup_{y \in X} E_{y}[\exp(tY_{0, M})]\right)^{n/(M+1)}\\
&= \left(1+ (\exp(t)-1) \sup_{y \in X} P_{y}(T_{y}^{+} > M)\right)^{n/(M+1)}.
\end{align*}

By noting the definition of $M$ and $F_{1}$, 
\begin{align*} 
\sup_{y \in X} P_{y}(T_{y}^{+} > M) 
&\le \sup_{y \in X} P_{y}(M < T_{y}^{+} < +\infty) + \sup_{y \in X} P_{y}(T_{y}^{+} = +\infty) \\
&\le \frac{\epsilon}{4} + 1-F_{1}.
\end{align*}
Hence, for any $t \ge 0$ and $x \in X$, 
\[ E_{x}\left[\exp\left(t\sum_{i \equiv a \!\!\!\!\mod(M+1)}Y_{i, M} \right)\right]
\le \left(1+ (\exp(t)-1)\left(\frac{\epsilon}{4} + 1-F_{1}\right) \right)^{n/(M+1)}.\]

Hence, 
the right hand side of the inequality (2.4)
is less than or equal to 
\begin{equation*} 
\left[\exp\left(-t\left(1-F_{1} + \frac{\epsilon}{2}\right)\right)\left\{1+ (\exp(t)-1)\left(\frac{\epsilon}{4} + 1-F_{1}\right)\right\}\right]^{n/(M+1)}. 
\end{equation*}

It is easy to see that 
for sufficiently small $t_{1} = t_{1}(F_{1}, \epsilon)> 0$, 
\[ \left\{1+ (\exp(t_{1})-1)\left(\frac{\epsilon}{4} + 1-F_{1}\right)\right\}
< \exp\left(t_{1}\left(1-F_{1} + \frac{\epsilon}{2}\right)\right).\]

Thus we have (2.3) and this convergence is uniform with respect to $x$.
This completes the proof of (1.1).\\

Second,  we will show (1.2).
Let $\epsilon > 0$.
Let $M$ be a positive integer.

By a last exit decomposition, 
\begin{align*}
P_{x}(R_{n} \le n(1-F_{2} - \epsilon)) &= P_{x}(n-R_{n} \ge n(F_{2} + \epsilon))\\
&= P_{x}\left(\sum_{i=0}^{n-2} (1- Y_{i, n-1-i}) \ge n(F_{2} + \epsilon)\right)\\
&\le P_{x}\left(\sum_{i=0}^{n-2} (1- Y_{i, \infty}) \ge n(F_{2} + \epsilon)\right).
\end{align*}

Now we have $1- Y_{i, \infty} = 1-Y_{i, M} + Y_{i, M} - Y_{i, \infty}$
and
\begin{align} 
P_{x}\left(\sum_{i=0}^{n-2} (1- Y_{i, \infty}) \ge n(F_{2} + \epsilon)\right)
&\le P_{x}\left(\sum_{i=0}^{n-2} (1- Y_{i, M}) \ge n\left(F_{2} + \frac{\epsilon}{2}\right)\right) \notag \\
&+ P_{x}\left(\sum_{i=0}^{n-2} (Y_{i, M}- Y_{i, \infty}) \ge \frac{n\epsilon}{2}\right) \tag{2.5}.
\end{align}

We have that $Y_{i, M}- Y_{i, \infty}$  is  the indicator function of 
\[ \{S_{i} \ne S_{i+k} \text{ for any } 1 \le k \le M, \, S_{i} = S_{i+k} \text{ for some } k > M\},\]
and hence, $E_{x}[Y_{i, M} - Y_{i, \infty}] \le  \sup_{y \in X} P_{y}(M < T_{y}^{+} < +\infty)$.

Then for any $n$, 
\begin{align} 
P_{x}\left(\sum_{i=0}^{n-2} (Y_{i, M}- Y_{i, \infty}) \ge \frac{n\epsilon}{2}\right)
&\le \frac{2}{n\epsilon} \sum_{i=0}^{n-2} E_{x}[Y_{i, M} - Y_{i, \infty}] \notag \\
&\le \frac{2}{\epsilon} \sup_{y \in X} P_{y}(M < T_{y}^{+} < +\infty). \tag{2.6}
\end{align}

On the other hand, 
\[ P_{x}\left(\sum_{i=0}^{n-2} (1- Y_{i, M}) \ge n\left(F_{2} + \frac{\epsilon}{2}\right)\right)\]
\[ \le \sum_{a = 0}^{M} P_{x}\left(\sum_{i \equiv a \!\!\!\!\mod M+1} (1- Y_{i, M}) \ge \frac{n}{M+1}\left(F_{2} + \frac{\epsilon}{2}\right)\right).\]

By the Markov property of $\{S_{n}\}_{n}$,
we have that for any $t > 0$ and any $a \in \{0,1,\dots, M\}$, 
\[ P_{x}\left(\sum_{i \equiv a \!\!\!\!\mod M+1} (1- Y_{i, M}) \ge \frac{n}{M+1}\left(F_{2} + \frac{\epsilon}{2}\right)\right)\]
\[ \le \exp\left(-t \frac{n}{M+1}\left(F_{2} + \frac{\epsilon}{2}\right)\right) 
E_{x}\left[ \prod_{i \equiv a \!\!\!\!\mod M+1} \exp(t(1-Y_{i, M}))\right] \]
\[ \le \exp\left(-t \frac{n}{M+1}\left(F_{2} + \frac{\epsilon}{2}\right)\right) 
\left(\sup_{y \in X} E_{y}\left[\exp(t(1-Y_{0, M}))\right]\right)^{n/(M+1)}\]
\[ = \left[\exp\left(-t \left(F_{2} + \frac{\epsilon}{2}\right)\right) 
\left\{1 + (\exp(t) - 1)\sup_{y \in X} P_{y}(T_{y}^{+} \le M)\right\}\right]^{n/(M+1)}.\]

Since $\sup_{y \in X} P_{y}(T_{y}^{+} \le M) \le F_{2}$,  
we have that 
for sufficiently small $t_{2} = t_{2}(F_{2}, \epsilon)> 0$, 
\[ \exp\left(-t_{2} \left(F_{2} + \frac{\epsilon}{2}\right)\right)
\left\{1 + (\exp(t_{2}) - 1)\sup_{y \in X} P_{y}(T_{y}^{+} \le M)\right\} < 1.\]

Therefore for any $a \in \{0,1,\dots M\}$, 
\[ P_{x}\left(\sum_{i \equiv a \!\!\!\!\mod M+1} (1- Y_{i, M}) \ge \frac{n}{M+1}\left(F_{2} + \frac{\epsilon}{2}\right)\right) \to 0, n \to \infty.\]

Thus we see that 
\[ P_{x}\left(\sum_{i=0}^{n-2} (1- Y_{i, M}) \ge n\left(F_{2} + \frac{\epsilon}{2}\right)\right)
\to 0, n \to \infty. \tag{2.7} \]
This convergence is uniform with respect to $x$. 

By using (2.5), (2.6) and (2.7), 
we have  
\[ \limsup_{n \to \infty} P_{x}(R_{n} \le n(1-F_{2} - \epsilon)) 
\le \frac{2}{\epsilon} \sup_{y \in X}P_{y}(M < T_{y}^{+} < +\infty).\]

By letting $M \to \infty$,
it follows from Lemma 2.1
that \[ \limsup_{n \to \infty} P_{x}(R_{n} \le n(1-F_{2} - \epsilon)) =0.\]
This convergence is uniform with respect to $x$. 
This completes the proof of (1.2). 
\end{proof}

\begin{Rem}
If $F_{1} = F_{2}$,
then (1.2) is easy to see by noting  (1.1) and  $E_{x}[R_{n}] \ge n(1-F_{2})$, $n \ge 1$, $x \in X$.
\end{Rem}

\begin{Cor}
If $\sup_{x} P_{x}(M < T_{x}^{+} < +\infty) = O(M^{-1-\delta})$ for some $\delta > 0$,
then, certain strong laws hold. 
More precisely, for any $x \in X$,  
\[ 1-F_{2} \le \liminf_{n \to \infty} \frac{R_{n}}{n} \le \limsup_{n \to \infty} \frac{R_{n}}{n} \le 1-F_{1}, \, P_{x}\text{-a.s}.\]
\end{Cor}

\begin{proof}
By noting the Borel-Cantelli lemma,
we see that it suffices to show that
for any $x \in X$ and $\epsilon > 0$, 
\[ \sum_{n \ge 1} P_{x}(R_{n} \ge n(1-F_{1} + \epsilon)) < +\infty,  \tag{2.8}\]
and, 
\[ \sum_{n \ge 1} P_{x}(R_{n} \le n(1-F_{2} - \epsilon)) < +\infty. \tag{2.9} \]

(2.8) follows from that the convergence (1.1) is exponentially fast.

By noting (2.5), (2.6) and (2.7),
we have that there exists $a = a(F_{2}, \epsilon) \in (0,1)$ such that
for any $n$ and $M < n$,   
\[ P_{x}(R_{n} \le n(1-F_{2} - \epsilon)) \le \frac{2}{\epsilon} O(M^{-1-\delta}) + a^{n/(M+1)}. \]
If we let $M = n^{1-\delta/2} -1$ for each $n$,
then, 
we see (2.9).
\end{proof}

Since the convergence in (1.1) is exponentially fast,
we can extend Theorem 1 in \cite{BIK},
which treats the range of the random walk bridge on vertex transitive graphs.  

\begin{Cor}
Let $(X, \mu)$ be an weighted graph satisfying $(U)$.
Let $x \in X$. 
We assume  that $\limsup_{n \to \infty} P_{x}(S_{2n} = x)^{1/n} = 1$.
Let $\epsilon > 0$.
Then,
\[ \lim_{n \to \infty} P_{x}\left(R_{n} \ge n(1-F_{1}+\epsilon)|S_{n} = x\right) = 0.\]
The limit is taken on $n$ such that $P_{x}(S_{n} = x) > 0$.
This convergence is exponentially fast.
\end{Cor}


\section{Proof of Theorem 1.3}

To begin with, 
we state a very rough sketch of the proof. 

Let $N_{1}, N_{2}$ be integers such that $3 < N_{1} < N_{2} < (N_{1}-1)^{2}$. 
First, we prepare a finite tree with degree $N_{1}$ and denote it $X^{(1)}$.
Second, we surround $X^{(1)}$ with finite trees with degree $N_{2}$.
We denote the graph we obtain by $X^{(2)}$.
Third, we surround $X^{(2)}$ with finite trees with degree $N_{1}$.
We denote the graph we obtain by $X^{(3)}$.
Repeating this construction,
we obtain an increasing sequence of finite trees $(X^{(n)})_{n}$.
$X^{(2n+1)} \setminus X^{(2n)}$ (resp. $X^{(2n+2)} \setminus X^{(2n+1)}$) is a ring-like object consisting of the $N_{1}$ (resp. $N_{2}$) -trees. 
Let $r_{2n+1}$ (resp. $r_{2n+2}$) be the ``width" of the ring. 
Assume $r_{i} \ll r_{i+1}$ for any $i$. 
Let $X$ be the infinite graph of the limit of $(X^{(n)})_{n}$.
This satisfies $(U)$, because $N_{1}$ and $N_{2}$ are not too far apart.
Lemma 3.3 states this formally. 
$X$ also satisfies $F_{1} < F_{2}$ and (1.3), because $r_{i} \ll r_{i+1}$ for any $i$.

In this section, 
we assume that any weight is equal to $1$, that is, $\mu_{xy} = 1$ for any $x \sim y$.

Let $X$ be an infinite tree.
For a connected subgraph $Y$ of $X$, 
we denote the restriction of 
$\mathcal{E}$, $\deg$, and $\rho$ to $Y$ 
by $\mathcal{E}_{Y}$, $\deg_{Y}$, and $\rho_{Y}$ respectively.
For a connected subgraph $Y \subset X$,
we let 
$\text{diam}(Y) = \sup_{y_{1}, y_{2} \in Y} d(y_{1}, y_{2})$.
Here $d$ is the graph distance on $X$.

Let $x \in X$.
Let \[ D_{x}(y) = \left\{z \in X : \text{the path between } x \text{ and } z \text{ contains } y\right\}, \, y \in X.\]
We remark that $y \in D_{x}(y)$ and $D_{x}(x) = X$.
Let 
$I_{x}(y, n) 
= \rho_{D_{x}(y)}(y,n)^{-1}$, $y \in X$. 
We remark that $I_{x}(x, n) = \rho_{X}(x,n)^{-1} = \rho(x,n)^{-1}$.
Then we have the following.

\begin{Lem}
Let $X$ be an infinite tree. 
Let $x, y \in X$.
Let $n \ge 1$.
Let $y_{i}$, $1 \le i \le \deg_{D_{x}(y)}(y)$, be the neighborhoods of $y$ in $D_{x}(y)$. 
Then,
\[ I_{x}(y, n+1) = \sum_{i=1}^{\deg_{D_{x}(y)}(y)}
\frac{ I_{x}(y_{i}, n)}{1+  I_{x}(y_{i}, n)}.\]
\end{Lem}

\begin{proof}
Let $f : D_{x}(y) \to \mathbb{R}$ such that
$f(y) = 1$ and $f = 0$ on $D_{x}(y) \setminus B_{D_{x}(y)}(y, n+1)$.
Then,
$f = 0$ on $D_{x}(y_{i}) \setminus B_{D_{x}(y_{i})}(y, n)$ 
for any $1 \le i \le \deg_{D_{x}(y)}(y)$.
Hence,
\begin{align*}
\mathcal{E}_{D_{x}(y)}(f, f) 
&= \sum_{i = 1}^{\deg_{D_{x}(y)}(y)} (1-f(y_{i}))^{2} + \mathcal{E}_{D_{x}(y_{i})}(f, f)\\
&\ge \sum_{i = 1}^{\deg_{D_{x}(y)}(y)} (1-f(y_{i}))^{2} + f(y_{i})^{2} I_{x}(y_{i}, n) \ge  \sum_{i = 1}^{\deg_{D_{x}(y)}(y)} \frac{I_{x}(y_{i}, n)}{1+I_{x}(y_{i}, n)}.
\end{align*}

Thus we see that
\[ I_{x}(y, n+1) \ge \sum_{i=1}^{\deg_{D_{x}(y)}(y)}
\frac{ I_{x}(y_{i}, n)}{1+  I_{x}(y_{i}, n)}.\]

Let $f_{i} : D_{x}(y_{i}) \to \mathbb{R}$ be a function such that $f_{i}(y_{i}) = 1$ and $f_{i} = 0$ 
on $ D_{x}(y_{i}) \setminus B_{ D_{x}(y_{i}) } (y_{i}, n)$, $1 \le i \le \deg_{D_{x}(y)}(y)$. 
Let $f  : D_{x}(y) \to \mathbb{R}$ be the function defined by $f(y) = 1$ and $f = f_{i}/(1+ \mathcal{E}_{D_{x}(y_{i})}(f_{i}, f_{i}))^{1/2}$ on $D_{x}(y_{i})$.
Then, $f = 0$ on $D_{x}(y) \setminus B_{D_{x}(y)}(y, n)$
and,
\[ I_{x}(y, n+1) \le \mathcal{E}_{D_{x}(y)}(f, f) = \sum_{i=1}^{\deg_{D_{x}(y)}(y)} \frac{\mathcal{E}_{D_{x}(y_{i})}(f_{i}, f_{i})}{1+\mathcal{E}_{D_{x}(y_{i})}(f_{i}, f_{i})}.\]

Since each $f_{i}$ is taken arbitrarily,
we have
\[
I_{x}(y, n+1) \le \sum_{i=1}^{\deg_{D_{x}(y)}(y)}
\frac{ I_{x}(y_{i}, n)}{1+  I_{x}(y_{i}, n)}. 
\]
These complete the proof of Lemma 3.1.
\end{proof}

\begin{Lem}
Let $3 \le N_{1} < N_{2}$.
Let $X$ be an infinite tree such that $\deg(x) \in [N_{1},N_{2}]$ for any $x \in X$.
Then, 
$N_{1}-2 \le I_{x}(y, n) \le N_{2}$
for any $x, y \in X$ and any $n \ge 1$.
\end{Lem}

\begin{proof}
We show this assertion by induction on $n$.
If $n = 1$,
then, by noting the definition of $D_{x}(y)$ and $I_{x}(y,1)$,  
$I_{x}(y, 1) = \deg_{D_{x}(y)}(y) \in [N_{1}-1, N_{2}]$. 
Thus the assertion holds.

We assume that $N_{1}-2 \le I_{x}(y, n) \le N_{2}$ for any $x, y \in X$.

Let $x, y \in X$.
Since $I_{x}(y, n+1) \le I_{x}(y, n)$,
we have
$I_{x}(y, n+1) \le N_{2}$.
Let $y_{i}$, $1 \le i \le \deg_{D_{x}(y)}(y)$,
be the neighborhoods of $y$ in $D_{x}(y)$.
By noting Lemma 3.1 and the assumption of induction,
\[ I_{x}(y, n+1) 
= \sum_{i=1}^{\deg_{D_{x}(y)}(y)} \frac{I_{x}(y_{i}, n)}{I_{x}(y_{i}, n) + 1} 
\ge \deg_{D_{x}(y)}(y)\frac{N_{1}-2}{N_{1}-1} 
\ge N_{1}-2.\]

These complete the proof of Lemma 3.2.
\end{proof}

\begin{Lem}
Let $3 \le N_{1} < N_{2} < (N_{1} - 1)^{2}$.
Let $X$ be an infinite tree such that $\deg(x) \in [N_{1},N_{2}]$ for any $x \in X$.
Then, $X$ satisfies $(U)$.
\end{Lem}

\begin{proof}
By using Lemma 3.1 and Lemma 3.2,
we have that for any $n, k \ge 1$ and
any $x, y \in X$, 
\begin{align*} 
I_{x}(y, n+1) - I_{x}(y, n+k+1) 
&= \sum_{i=1}^{\deg_{D_{x}(y)}(y)} \frac{I_{x}(y_{i}, n)}{I_{x}(y_{i}, n)+1} - \frac{I_{x}(y_{i}, n+k)}{I_{x}(y_{i}, n+k) + 1}\\
&\le \frac{N_{2}}{(N_{1}-1)^{2}} 
\sup_{z \in X} \left(I_{x}(z, n) - I_{x}(z, n+k)\right). 
\end{align*}
Here $y_{i}$, $1 \le i \le \deg_{D_{x}(y)}(y)$, be the neighborhoods of $y$
in $D_{x}(y)$.

Repeating this argument, we have that for any $n, k \ge 1$,
\begin{align*} 
I_{x}(x, n) - I_{x}(x, n+k) 
&\le \left(\frac{N_{2}}{(N_{1}-1)^{2}}\right)^{n-1} 
\sup_{z \in X} (I_{x}(z, 1) - I_{x}(z, k+1))\\
&\le \left(\frac{N_{2}}{(N_{1}-1)^{2}}\right)^{n-1} N_{2}. 
\end{align*}
Since $N_{2} < (N_{1}-1)^{2}$, 
$\rho_{X}(x,n)^{-1}$ converges uniformly to $\rho_{X}(x)^{-1}$, $n \to \infty$.

By Lemma 3.2,
$\rho_{X}(x, n)^{-1} \ge N_{1}-2$ for any $n \ge 1$ 
and hence $\rho_{X}(x)^{-1} \ge N_{1}-2$.
Therefore, 
\begin{align*} 
\rho_{X}(x) - \rho_{X}(x,n) &= \rho_{X}(x)\rho_{X}(x,n)(\rho_{X}(x,n)^{-1} - \rho_{X}(x)^{-1} ) \\
&\le \frac{\rho_{X}(x,n)^{-1} - \rho_{X}(x)^{-1}}{(N_{1}-2)^{2}}.  
\end{align*} 
Hence
$\rho_{X}(x,n)$ converges uniformly to $\rho_{X}(x)$, $n \to \infty$.
This completes the proof of Lemma 3.3.
\end{proof}

Let $N \ge 3$. 
Let $T_{N}$ be the infinite $N$-regular tree. 
Let $\tilde T_{N}(o)$ be the infinite tree $T$ such that $\deg(o) = N-1$ for $o \in T$ and $\deg(x) = N$ for any $x \in T \setminus \{o\}$.
For the simple random walk on $T_{N}$,
we let $g_{N} = P_{x}(T_{x}^{+} = +\infty)$
and $g_{N}(n) = P_{x}(T_{x}^{+} > n)$ for some (or any) $x \in T_{N}$.

\begin{Def}
Let $Y$ be a finite tree.
Let $E(Y) = \{y \in Y : \deg(y) = 1\}$.
Let $N \ge 3$.
We define an infinite tree $Y_{N}$ as follows : 
We prepare $Y$ and $|E(Y)|$ copies of $\tilde T_{N}(o)$.
Let $Y_{N}$ be the infinite tree obtained by attaching $o \in \tilde T_{N}(o)$ to each $y \in E(Y)$.
\end{Def}

\begin{Lem}
Let $N \ge 3$. 
Let $Y$ be a finite tree with a reference point $o$ 
such that $\deg(y) \ge 3$ for any $y \in Y \setminus E(Y)$.
Let $Y_{N}$ be the infinite tree in Definition 3.4.
We assume that $Y_{N}$ satisfies $(U)$.
Let $R_{n}$ be the range of the simple random walk up to time $n-1$ on $Y_{N}$.
Then, 
\[ \lim_{n \to \infty} \frac{E_{o}[R_{n}]}{n} = g_{N}.\]
\end{Lem}

\begin{proof}
By considering a last exit decomposition as in the proof of Theorem 1.2, 
\begin{align*} 
E_{o}[R_{n}]
&=  1 + \sum_{i=0}^{n-2} P_{o}\left(S_{i} \ne S_{j} \text{ for any }  j \in \{i+1, \dots, n-1\}\right)  \\
&=  1 + \sum_{i=0}^{n-2} \sum_{y \in Y_{N}} P_{o}(S_{i} = y)P_{y}(T_{y}^{+} > n-1-i)\\
&=  1 + \sum_{i=0}^{n-2} \sum_{y \in Y_{N}} P_{o}(S_{i} = y)P_{y}(n-1-i < T_{y}^{+} < +\infty) \\
& \, \, \, \, + \sum_{i=0}^{n-2} \sum_{y \in Y_{N}} P_{o}(S_{i} = y)P_{y}(T_{y}^{+} = +\infty).
\end{align*}

Since $Y_{N}$ satisfies $(U)$,
\[ \frac{1}{n} \sum_{i=0}^{n-2} \sum_{y \in Y_{N}} P_{o}(S_{i} = y)P_{y}(n-1-i < T_{y}^{+} < +\infty) \]
\[ \le \frac{1}{n} \sum_{i=0}^{n-2} \sup_{y \in Y_{N}} P_{y}(n-1-i < T_{y}^{+} < +\infty) \to 0, \, n \to \infty.\]

Hence it is sufficient to show that
\[ \lim_{n \to \infty} \frac{1}{n} \sum_{i=0}^{n-2} \sum_{y \in Y_{N}} P_{o}(S_{i} = y)P_{y}(T_{y}^{+} = +\infty) = g_{N}. \tag{3.1} \]

By the assumption,
$Y_{N}$ is an infinite tree such that
$\deg(y) \ge 3$ for any $y \in Y_{N}$
and
$\sup_{y \in Y_{N}} \deg(y) < +\infty$.
Then, by Woess \cite{W} Example 3.8,
$Y_{N}$ is roughly isometric to the $3$-regular tree $T_{3}$.
Therefore $Y_{N}$ is a transient graph.

Let $x, y \in Y_{N}$.
Since $P_{x}(S_{i} = y)/\deg(y) = P_{y}(S_{i} = x)/\deg(x)$, 
\[ \frac{\deg(x)}{\deg(y)} P_{x}(S_{i} = y)^{2} = P_{x}(S_{i} = y)P_{y}(S_{i} = x) \le P_{x}(S_{2i} = x). \]
Therefore, 
\begin{equation}
P_{x}(S_{i} = y) \to 0, \,  i \to \infty, \text{ for any } x, y \in Y_{N}. \tag{3.2}
\end{equation}

Let $\epsilon > 0$.
Then,
there exists a positive integer $m_{0}$
such that
$g_{N}(m_{0}) \le g_{N} + \epsilon/2$.
By using the definition of $Y_{N}$ and 
that the distribution of the random walk up to time $n-1$ starting at $y \in Y_{N}$ is determined by $B_{Y_{N}}(y, n)$, 
we have that 
\[ P_{y}(T_{y}^{+} > m_{0}) = g_{N}(m_{0})  \, \text{ for any } y \in Y_{N} \setminus B(o, 2(\text{diam}(Y) + m_{0})).  \]

Hence
$P_{y}(T_{y}^{+} = +\infty) \le g_{N} + \epsilon/2$ 
for any $y \in Y_{N} \setminus B(o, 2(\text{diam}(Y) + m_{0}))$.

By (3.2), we have that
$P_{o}(S_{i} \in B(o, 2(\text{diam}(Y) + m_{0}))) \to 0$,
$i \to \infty$.
Hence there exists a positive integer $n_{0}$ 
such that 
$P_{o}(S_{i} \in B(o, 2(\text{diam}(Y) + m_{0}))) \le \epsilon/2$ for any $i \ge n_{0}$.
Then, for any $n > n_{0}$,

\[ \frac{1}{n} \sum_{i = n_{0}}^{n-1} \sum_{y \in Y_{N}}P_{o}(S_{i} = y)P_{y}(T_{y}^{+} = +\infty)\]
\[ \le \frac{1}{n} \sum_{i = n_{0}}^{n-1} 
P_{o}(S_{i} \in  B(o, 2(\text{diam}(Y) + m_{0}))) +\sup_{y \notin B(o, 2(\text{diam}(Y) + m_{0}))} P_{y}(T_{y}^{+} = +\infty). \]
\[ \le \frac{n-n_{0}}{n} \frac{\epsilon}{2} + g_{N} + \frac{\epsilon}{2} \le g_{N} + \epsilon. \]

We remark that  
\[ \frac{1}{n} \sum_{i = 0}^{n_{0}-1} \sum_{y \in Y_{N}}P_{o}(S_{i} = y)P_{y}(T_{y}^{+} = +\infty) \le \frac{n_{0}}{n} \to 0, \, n \to \infty.\]

Since $\epsilon > 0$ is taken arbitrarily,
we see that
\[ \limsup_{n \to \infty} \frac{1}{n} \sum_{i = 0}^{n-1} \sum_{y \in Y_{N}}P_{o}(S_{i} = y)P_{y}(T_{y}^{+} = +\infty) \le g_{N}.\tag{3.3} \]

Let $\epsilon > 0$. 
Since $Y_{N}$ satisfies $(U)$,
there exists a positive integer $m_{1}$
such that $\sup_{y \in Y_{N}}P_{y}(m_{1} < T_{y}^{+} < +\infty) \le \epsilon$.
We have that for any $y \in Y_{N} \setminus B(o, 2(\text{diam}(Y) + m_{1})$, 
$P_{y}(T_{y}^{+} > m_{1}) = g_{N}(m_{1})$.
Hence,
\[ g_{N} \le g_{N}(m_{1}) = P_{y}(T_{y}^{+} = +\infty) + P_{y}(m_{1} < T_{y}^{+} < +\infty) \le P_{y}(T_{y}^{+} = +\infty) + \epsilon \]
for any $y \in Y_{N} \setminus B(o, 2(\text{diam}(Y) + m_{1}))$. 

By (3.2), 
there exists a positive integer $n_{1}$ 
such that 
$P_{o}(S_{i} \in B(o, 2(\text{diam}(Y) + m_{1}))) \le \epsilon$, for any $i \ge n_{1}$.
Then, 
\[  \frac{1}{n} \sum_{i = 0}^{n-1} \sum_{y \in Y_{N}}P_{o}(S_{i} = y)P_{y}(T_{y}^{+} = +\infty) \]
\[ \ge \frac{1}{n} \sum_{i = n_{1}}^{n-1} \sum_{y \notin B(o, 2(\text{diam}(Y) + m_{1}))}P_{o}(S_{i} = y)P_{y}(T_{y}^{+} = +\infty) 
\ge \frac{n-n_{1}}{n}(1-\epsilon) (g_{N} - \epsilon). \]

By letting $n \to \infty$ and recalling that $\epsilon > 0$ is taken arbitrarily,
\begin{equation} 
\liminf_{n \to \infty} \frac{1}{n} \sum_{i = 0}^{n-1} \sum_{y \in Y_{N}}P_{o}(S_{i} = y)P_{y}(T_{y}^{+} = +\infty) \ge g_{N}. \tag{3.4}
\end{equation}

(3.3) and (3.4) imply (3.1).
\end{proof}

\begin{proof}[Proof of Theorem 1.3]
First, we will construct an increasing sequence of finite trees $(X^{(n)})_{n}$ by induction on $n$.
Second, we will show that the limit infinite graph $X$ of $(X^{(n)})_{n}$ satisfies $(U)$, $F_{1} < F_{2}$ and (1.3).

Let $3 \le N_{1} < N_{2} < (N_{1}-1)^{2}$.
Let $X^{(1)}$ be a finite tree 
such that
$\deg(x) = N_{1}$ for any $x \in X^{(1)} \setminus E(X^{(1)})$
and
$X^{(1)} = B(o, k_{1})$ for a point $o \in X^{(1)}$ and a positive integer $k_{1}$.

We assume that $X^{(2n-1)}$ is constructed and 
$X^{(2n-1)} = B_{X^{(2n-1)}}(o, k_{2n-1})$ 
for a positive integer $k_{2n-1}$.
By Lemma 3.5,
there exists $k_{2n} > 2k_{2n-1}$
such that
for the simple random walk on $(X^{(2n-1)})_{N_{2}}$ starting at $o$, 
\begin{equation}
\frac{E_{o}[R_{k_{2n}}]}{k_{2n}} \ge g_{N_{2}} - \frac{1}{n}.\tag{3.5}
\end{equation}

Then we let $X^{(2n)} = (X^{(2n-1)})_{N_{2}} \cap B_{(X^{(2n-1)})_{N_{2}}}(o, k_{2n})$.

We assume that $X^{(2n)}$ is constructed and 
$X^{(2n)} = B_{X^{(2n)}}(o, k_{2n})$ 
for a positive integer $k_{2n}$.
By Lemma 3.5,
there exists $k_{2n+1} > 2k_{2n}$
such that
for the simple random walk on $(X^{(2n)})_{N_{1}}$ starting at $o$, 
\begin{equation}
\frac{E_{o}[R_{k_{2n+1}}]}{k_{2n+1}} \le g_{N_{1}} + \frac{1}{n}.\tag{3.6}
\end{equation}

Then we let $X^{(2n+1)} = (X^{(2n)})_{N_{1}} \cap B_{(X^{(2n)})_{N_{1}}}(o, k_{2n+1})$.

Let $X$ be the infinite graph obtained by the limit of a sequence of $(X^{(n)})$.
Then $\deg_{X}(x) \in \{N_{1}, N_{2}\}$ and  by Lemma 3.3 $X$ satisfies $(U)$.

Now we show (1.3). 
We remark that the distribution of the simple random walk up to time $k-1$ on $X$ starting at $o$ is determined by $B_{X}(o, k)$, $k \ge 1$.
By the definition of $X$, 
(3.5) and (3.6) hold also for the simple random walk on $X$.
Hence, 
\[ \liminf_{n \to \infty} \frac{E_{o}[R_{n}]}{n} \le g_{N_{1}}, 
\mathrm{ and, }
\limsup_{n \to \infty} \frac{E_{o}[R_{n}]}{n} \ge g_{N_{2}}. \tag{3.7}\]

By considering a last exit decomposition as in the proof of Theorem 1.2, 
and,  
noting that $X$ satisfies $(U)$,
we have 
\[ 1 - F_{2} 
= \inf_{x \in X} P_{x}(T_{x}^{+} = +\infty) 
\le \liminf_{n \to \infty} \frac{E_{o}[R_{n}]}{n}, \tag{3.8}\]
and, 
\[ \limsup_{n \to \infty} \frac{E_{o}[R_{n}]}{n} 
\le \sup_{x \in X} P_{x}(T_{x}^{+} = +\infty)
= 1 - F_{1}. \tag{3.9}\]

In order to see (1.3), 
it is sufficient to show that for any  $x \in X$, 
\[ g_{N_{1}} \le  P_{x}(T_{x}^{+} = +\infty) \le g_{N_{2}}. \tag{3.10} \]
 
Let $x \in X$. 
We recall that $\deg_{X}(x) = N_{1}$ or $\deg_{X}(x) = N_{2}$.
Then we can assume 
that $T_{N_{1}}$ is a subtree of $X$ 
and 
$X$ is a subtree of $T_{N_{2}}$ 
and 
$x \in T_{N_{1}}$.

Assume $\deg_{X}(x) = N_{1}$.  
By using \cite{Kum} Theorem 1.16 and that $T_{N_{1}}$ is a subtree of $X$, 
we have  
\[ g_{N_{1}} 
= N_{1}^{-1} \rho_{T_{N_{1}}}(x)^{-1} 
\le N_{1}^{-1} \rho_{X}(x)^{-1} 
= P_{x}(T_{x}^{+} = +\infty).\]

Let $x_{i}$, $1 \le i \le N_{2}$, be the neighborhoods of $x$ in $T_{N_{2}}$ and $x_{i} \in T_{N_{1}}$ for $1 \le i \le N_{1}$.
Let $f : T_{N_{2}} \to \mathbb{R}$ be a function 
such that $f(x) = 1$ and 
it has a compact support in $T_{N_{2}}$.
Then, $f$ has compact support also in $D_{x}(x_{i})$ for each $i$.
Here $D_{x}(x_{i})$ is defined in $T_{N_{2}}$.
Then, 
\begin{align*} 
\mathcal{E}_{T_{N_{2}}}(f, f) - \mathcal{E}_{X}(f, f)
&\ge \sum_{i=N_{1} + 1}^{N_{2}} (1-f(x_{i}))^{2}  + \mathcal{E}_{D_{x}(x_{i})}(f,f)\\
&\ge \sum_{i=N_{1} + 1}^{N_{2}} (1-f(x_{i}))^{2}  + f(x_{i})^{2}
\rho_{D_{x}(x_{i})}(x_{i})^{-1} \\
&\ge \sum_{i=N_{1} + 1}^{N_{2}} \frac{\rho_{D_{x}(x_{i})}(x_{i})^{-1}}{1+\rho_{D_{x}(x_{i})}(x_{i})^{-1}}.
\end{align*}

Since $D_{x}(x_{i})$ is graph isomorphic to $\tilde T_{N_{2}}(o)$,
$\rho_{D_{x}(x_{i})}(x_{i}) = \rho_{\tilde T_{N_{2}}(o)}(o)$.
Hence, 
\[ \mathcal{E}_{T_{N_{2}}}(f, f) - \mathcal{E}_{X}(f, f) 
\ge (N_{2} - N_{1}) \frac{\rho_{\tilde T_{N_{2}}(o)}(o)^{-1}}{1+ \rho_{\tilde T_{N_{2}}(o)}(o)^{-1}}.\]

Since $f|_{X}(x) = 1$ and $f|_{X}$ has compact support on $X$, 
\[  \mathcal{E}_{T_{N_{2}}}(f, f) - \rho_{X}(x)^{-1} 
\ge (N_{2} - N_{1}) \frac{\rho_{\tilde T_{N_{2}}(o)}(o)^{-1}}{1+ \rho_{\tilde T_{N_{2}}(o)}(o)^{-1}}.\]

Since $f$ is taken arbitrarily,   
\[ \rho_{T_{N_{2}}}(x)^{-1} - \rho_{X}(x)^{-1} 
\ge (N_{2} - N_{1}) \frac{\rho_{\tilde T_{N_{2}}(o)}(o)^{-1}}{1+ \rho_{\tilde T_{N_{2}}(o)}(o)^{-1}}.\]

We see that
$\rho_{T_{N_{2}}}(x)^{-1} = N_{2} \dfrac{\rho_{\tilde T_{N_{2}}(o)}(o)^{-1}}{1+ \rho_{\tilde T_{N_{2}}(o)}(o)^{-1}}$
in the same manner as in the proof of Lemma 3.1.
Hence, $N_{1} \rho_{T_{N_{2}}}(x)^{-1} \ge N_{2} \rho_{X}(x)^{-1}$. 
By using \cite{Kum} Theorem 1.16,  
we see that 
\[ P_{x}(T_{x}^{+} = +\infty) = N_{1}^{-1}\rho_{X}(x)^{-1}
\le N_{2}^{-1}\rho_{T_{N_{2}}}(x)^{-1}= g_{N_{2}}. \]

Assume $\deg_{X}(x) = N_{2}$. 
We can show (3.10) in the same manner as above and sketch the proof.  

By using \cite{Kum} Theorem 1.16 and that $X$ is a subtree of $T_{N_{2}}$, 
we have  
\[ P_{x}(T_{x}^{+} = +\infty) 
= N_{2}^{-1} \rho_{X}(x)^{-1} 
\le N_{2}^{-1} \rho_{T_{N_{2}}}(x)^{-1}
= g_{N_{2}}. \]

Let $x_{i}$, $1 \le i \le N_{2}$, be the neighborhoods of $x$ in $X$ 
and $x_{i} \in T_{N_{1}}$ for $1 \le i \le N_{1}$.
Let $f : X \to \mathbb{R}$ be a function 
such that $f(x) = 1$ and 
it has a compact support in $X$.
Then, $f$ has compact support also in $D_{x}(x_{i})$ for each $i$.
Here $D_{x}(x_{i})$ is defined in $X$.
Then, 
\begin{align*} 
\mathcal{E}_{X}(f, f) - \mathcal{E}_{T_{N_{1}}}(f, f)
&\ge \sum_{i=N_{1} + 1}^{N_{2}} (1-f(x_{i}))^{2}  + \mathcal{E}_{D_{x}(x_{i})}(f,f)\\
&\ge \sum_{i=N_{1} + 1}^{N_{2}} \frac{\rho_{D_{x}(x_{i})}(x_{i})^{-1}}{1+\rho_{D_{x}(x_{i})}(x_{i})^{-1}}.
\end{align*}
We can regard $\tilde T_{N_{1}}(o)$ as a subtree of $D_{x}(x_{i})$ and 
can assume $x_{i} = o$. 
Hence $\rho_{D_{x}(x_{i})}(x_{i})^{-1} 
\ge \rho_{\tilde T_{N_{1}}(o)}(o)^{-1}$ and 
\[ \mathcal{E}_{X}(f, f) - \mathcal{E}_{T_{N_{1}}}(f, f) 
\ge (N_{2}-N_{1}) \frac{\rho_{\tilde T_{N_{1}}(o)}(o)^{-1}}{1+ \rho_{\tilde T_{N_{1}}(o)}(o)^{-1}}. \]

Therefore,
\[ \rho_{X}(x)^{-1} - \rho_{T_{N_{1}}}(x)^{-1} 
\ge (N_{2}-N_{1}) \frac{\rho_{\tilde T_{N_{1}}(o)}(o)^{-1}}{1+ \rho_{\tilde T_{N_{1}}(o)}(o)^{-1}} 
= (N_{2} - N_{1}) \frac{\rho_{T_{N_{1}}}(x)^{-1} }{N_{1}}. \]
and then we have $N_{1} \rho_{X}(x)^{-1}  \ge N_{2} \rho_{T_{N_{1}}}(x)^{-1}$.
By using \cite{Kum} Theorem 1.16,  
we see that 
\[ g_{N_{1}} 
= N_{1}^{-1}\rho_{T_{N_{1}}}(x)^{-1}
\le  N_{2}^{-1}\rho_{X}(x)^{-1} 
= P_{x}(T_{x}^{+} = +\infty). \]

Thus the proof of (3.10) completes and we obtain (1.3).

By using \cite{W} Lemma 1.24 and $N_{1} < N_{2}$, 
we see that 
$g_{N_{1}} = (N_{1}-2)/(N_{1}-1) < g_{N_{2}} = (N_{2}-2)/(N_{2}-1)$. 
By using (3.7), (3.8), (3.9) and (3.10),
we see $g_{N_{1}} = 1 - F_{2}$ and $g_{N_{2}} = 1 - F_{1}$.
Hence $F_{1} < F_{2}$. 

Thus we see that $X$ satisfies $(U)$, $F_{1} < F_{2}$, and, (1.3). 
\end{proof}


\section{Examples of graphs satisfying the uniform condition}

In this section,
we give some examples of graphs satisfying $(U)$. 
We assume that all weights are equal to $1$. 

Here we follow \cite{Kum} Definition 1.8 for the definition of rough isometry introduced by Kanai \cite{Kan1}.

\begin{Def}
Let $X_{i}$ be weighted graphs and $d_{i}$ be the graph metric of $X_{i}$, $i=1,2$. 
We say that a map $T : X_{1} \to X_{2}$ is a ($(A, B, M)$-)\textit{rough isometry} 
if there exist constants $A > 1$, $B > 0$, and, $M > 0$ satisfying the following inequalities.
\[ A^{-1}d_{1}(x,y) - B \le d_{2}(T(x), T(y)) \le Ad_{1}(x,y) + B, \, \, x, y \in X_{1}. \]
\[ d_{2}(T(X_{1}), z) \le M, \, \, z \in X_{2}.  \]
We say that $X_{1}$ is \textit{roughly isometric} to $X_{2}$ 
if there exists a rough isometry between them. 
We say that a property is \textit{stable under rough isometry} 
if whenever $X_{1}$ satisfies the property and is roughly isometric to $X_{2}$, then $X_{2}$ also satisfies the property.
\end{Def}

\subsection{Recurrent graphs}

\begin{Prop}
The condition $(U)$ is stable under rough isometry between recurrent graphs.
\end{Prop}

\begin{proof}
Assume that $X_{1}$ is a recurrent graph satisfying $(U)$ and $X_{2}$ is a (recurrent) graph which is roughly isometric to $X_{1}$.
We would like to show that $X_{2}$ satisfies $(U)$.

Since rough isometry is an equivalence relation,
there exists a $(A, B, M)$-rough isometry $T : X_{2} \to X_{1}$.  
Fix $n \in \mathbb{N}$ and $x \in X_{2}$.
Let $f$ be a function on $X_{1}$
such that $f(T(x)) = 1$ and $f = 0$ on $X_{1} \setminus B(T(x), A^{-1}n-B)$.
Since $T$ is a $(A, B, M)$-rough isometry,
we have that for any $y \in X_{2} \setminus B(x, n)$,
$T(y) \in X_{1} \setminus B(T(x), A^{-1}n - B)$, 
and hence, $f \circ T = 0$ on $X_{2} \setminus B(x, n)$. 

By using Theorem 3.10 in \cite{W},
we see that
there exists a constant $c > 0$
such that
$\mathcal{E}_{X_{1}}(f, f) \ge c \mathcal{E}_{X_{2}}(f \circ T, f \circ T)$.
This constant does not depend on $(x,n,f)$.
Therefore, 
\[ \inf\left\{\mathcal{E}_{X_{1}}(f, f) :  f(T(x)) = 1, f = 0 \text{ on } X_{1} \setminus B(T(x), A^{-1}n-B) \right\} \]
\[ \ge c \inf\left\{ \mathcal{E}_{X_{2}}(g, g) : g(x) = 1, g = 0 \text{ on } X_{2} \setminus B(x, n) \right\}. \]
Hence, $\rho_{X_{2}}(x, n) \ge c \rho_{X_{1}}(T(x), A^{-1}n-B)$.
By recalling that $X_{1}$ satisfies $(U)$, 
we see that $X_{2}$ satisfies $(U)$. 
\end{proof}

\begin{Prop}
Let $X$ be a graph such that 
there exists $C > 0$ such that
$V(x,n) \le Cn^{2}$ for any $x \in X$ and $n \ge 1$.
Let $X^{\prime}$ be a graph which is roughly isometric to $X$.
Then, $X$ and $X^{\prime}$ satisfy $(U)$. 
\end{Prop}

We can show the above assertion in the same manner as in the proof of \cite{W}, Lemma 3.12 and Lemma 3.13,
so we omit the proof.

\begin{Prop}
Let $X$ be a graph such that 
\[ \lim_{n \to \infty} \inf_{x \in X} \sum_{k=0}^{n} p_{k}(x,x) = +\infty. \tag{4.1}\] 
Let $X^{\prime}$ be a graph which is roughly isometric to $X$.
Then, $X$ and $X^{\prime}$ satisfy $(U)$. 
\end{Prop}

\begin{proof}
By noting \cite{Kum} Lemma 3.3(iv),
we see that $\rho(x,n) = g^{B(x,n)}(x,x)$, $x \in X$, $n \ge 1$.
Since $p_{k}^{B(x,n)}(x, x) = p_{k}(x, x)$ for $k < n$,
\[ \rho(x,n) = g^{B(x,n)}(x,x) \ge \sum_{0 \le k < n} p_{k}(x,x). \]
By noting (4.1), we see $X$ satisfies $(U)$.
Since $X$ is recurrent,
it follows from Proposition 4.2 
that $X^{\prime}$ also satisfies $(U)$. 
\end{proof}

By using Section 5 in Barlow, Coulhon and Kumagai \cite{BCK}, 
we see that the $d$-dimensional standard graphical Sierpi\'nski gaskets, $d \ge 2$, and Vicsek trees (See Barlow \cite{B} for definition) satisfies (4.1). 
Thus we have 
\begin{Exa}
The graphs which are roughly isometric with the following graphs satisfy $(U)$.\\
(i)  Infinite connected subgraphs in $\mathbb{Z}^{2}$.\\
(ii)  Infinite connected subgraphs  in the planer triangular lattice.\\
(iii) The $d$-dimensional standard graphical Sierpi\'nski gaskets, $d \ge 2$. \\
(iv) Vicsek trees.
\end{Exa}

\subsection{Transient graphs}

\begin{Prop}
Assume that a graph $X$ satisfies $(UC_{\alpha})$, $\alpha > 2$, 
that is, 
there exist $C > 0$
such that $\sup_{x \in X} p_{n}(x, x) \le Cn^{-\alpha/2}$, $n \geq 1$.
Let $X^{\prime}$ be a graph which is roughly isometric to $X$.
Then, $X$ and $X^{\prime}$ satisfy $(U)$. 
\end{Prop}

\begin{proof}
Let $m > n$. 
Then, by using \cite{K} Lemma 3.3(iv) and $p_{k}^{B(x,m)}(x, x) = p_{k}^{B(x,n)}(x, x) = p_{k}(x, x)$ for $k < n$,
\begin{align*}
\rho(x,m) - \rho(x,n) &= g^{B(x,m)}(x,x) - g^{B(x,n)}(x,x) \\
&= \sum_{k \ge n} (p_{k}^{B(x,m)}(x, x) - p_{k}^{B(x,n)}(x, x)) \\
&\le \sum_{k \ge n} p_{k}(x, x).
\end{align*}

Letting $m \to \infty$,
\[ \rho(x) - \rho(x, n) \le  \sum_{k \geq n} p_{k}(x, x), x \in X, n \geq 1.\]
Thus we see that if $X$ satisfies $(UC_{\alpha})$ for some $\alpha > 2$, then $X$ satisfies $(U)$.

The stability of the property $(UC_{\alpha})$, $\alpha > 2$, under rough isometry follows from Varopoulous \cite{V} Theorem 1 and 2, and, Kanai \cite{Kan2} Proposition 2.1.
Thus we see that $X^{\prime}$ also satisfies $(U)$. 
\end{proof}

$\mathbb{Z}^{d}$ satisfies $(UC_{d})$.
By using Barlow and Bass \cite{BB1}, \cite{BB2},
we see that if $d \ge 3$, 
then $d$-dimensional standard graphical Sierpi\'nski carpet satisfies $(UC_{\alpha})$ for some $\alpha > 2$. 
Therefore we have 
\begin{Exa}
The graphs which are roughly isometric with the following graphs satisfy $(U)$.\\
(i) $\mathbb{Z}^{d}$, $d \ge 3$.\\
(ii) $d$-dimensional standard graphical Sierpi\'nski carpet, $d \ge 3$.
\end{Exa}

\subsection{A graph which does not satisfy $(U)$}

Finally, we give an example of a graph which does not satisfy $(U)$. 

\begin{Rem}
The recurrent tree $T$ treated in \cite{W}, Example 6.16 does not satisfy $(U)$. 
For any $n \geq 1$, there exists $x_{n} \in T$ such that $\rho(x_{n}, n) = \rho_{T_{4}}(x_{n}, n)$,
where $T_{4}$ is the $4$-regular tree.
Since  $T_{4}$ is vertex transitive and transient, 
we have that 
$\rho_{T_{4}}(x_{n}, n) \le \rho_{T_{4}}(x_{n}) = \rho_{T_{4}}(o) < +\infty$, 
$n \ge 1$, 
for a reference point $o \in T_{4}$.
However, $T$ is recurrent and hence $\rho(x, n) \to \infty$, $n \to \infty$, $x \in T$. 
Thus we see that $T$ does not satisfy $(U)$.
\end{Rem}

\textit{Acknowledgement} \, The author wishes to express his gratitude to Prof. K. Hattori for pointing out a flaw in a previous version of the paper.

\end{document}